\documentclass[12pt,letterpaper]{article}
\usepackage{mathptmx}       
\usepackage{helvet}         
\usepackage{courier}        
\usepackage{type1cm}        
\usepackage{makeidx}         
\usepackage{graphicx}        
\usepackage{multicol}        
\usepackage{amsmath}
\usepackage{amssymb}

\usepackage{amsthm}
\usepackage{multirow}
\usepackage{array}
\usepackage{tabularx}
\usepackage{filecontents}
\usepackage[utf8x]{inputenc}
\usepackage{subcaption}
\usepackage{graphicx}
\usepackage[numbers,square]{natbib}
\usepackage{xcolor}
\usepackage{dcolumn}
\usepackage{sectsty}
\usepackage{booktabs,  underscore, url, graphicx}
\usepackage[euler]{textgreek}
\newtheorem{thm}{Theorem}
\newtheorem{lem}[thm]{Lemma}

\newcommand{\llll}[1] {\left #1}
\newcommand{\rrrr}[1] {\right #1}

\newcommand{\dddd}[2]{\dfrac{#1}{#2}}

\newcommand{\aaaa}{\alpha}

\newcommand{\bbbb}{\beta}
\newcommand{\ssss}{\sigma}
\newcommand{\eeee}{\epsilon}
\newcommand{\GGGG}{\Gamma}
\newcommand{\gggg}{\gamma}
\newcommand{\zzzz}{\zeta}
\newcommand{\dddddd}{\delta}
\newcommand{\LLLL}{\lambda}

\makeindex             
\begin{document}
\title{Numerical solutions of  ordinary fractional differential equations with singularities}
\author{Yuri Dimitrov, Ivan Dimov, Venelin Todorov}
%
\date{}
\maketitle

\abstract{The solutions of fractional differential equations (FDEs) have a natural singularity at the initial point.  The accuracy of their numerical solutions is lower than the accuracy of the numerical solutions of FDEs whose solutions are differentiable functions. In the present paper we propose a method for improving the accuracy of the numerical solutions of  ordinary linear FDEs  with constant coefficients which uses the fractional Taylor polynomials of the solutions. The numerical solutions of the two-term and three-term  FDEs  are studied in the paper.}

\section{Introduction}
In recent years there is a growing interest in applying FDEs for modeling diffusion processes in biology, engineering and finance \cite{Cartea2007,Srivastava2014}. The two main approaches to fractional differentiation are  the Caputo and Riemann-Liouville fractional derivatives. The Caputo  derivative of order $\aaaa$, where $0<\alpha<1$ is defined as
\begin{equation*}\label{Cder}
D_C^\alpha y(t)=y^{(\alpha)}(t)=\dddd{d^\aaaa}{d t^\aaaa}y(t)=\dfrac{1}{\Gamma (1-\alpha)}\int_0^t \dfrac{y'(\xi)}{(t-\xi)^\alpha}d\xi.
\end{equation*}
The Caputo derivative is a standard choice for a fractional derivative in the models using FDEs. The finite difference schemes for numerical solution of FDEs involve approximations of the fractional derivative.
 Let $t_n=n h$, where $h$ is a small positive number and $y_n=y(t_n)=y(n h).$ The L1 approximation  is an important and  commonly used approximation of the  Caputo derivative.
 \begin{equation} \label{1a1}
y^{(\alpha)}_n = \dfrac{1}{h^\alpha \GGGG(2-\aaaa)}\sum_{k=0}^{n} \ssss_k^{(\alpha)} y_{n-k}+O\llll(h^{2-\aaaa}\rrrr),
\end{equation}
where $\ssss_0^{(\alpha)}=1,\; \ssss_n^{(\alpha)}=(n-1)^{1-\aaaa}-n^{1-\aaaa}$ and
$$\ssss_k^{(\alpha)}=(k+1)^{1-\alpha}-2k^{1-\alpha}+(k-1)^{1-\alpha},\quad (k=1,\cdots,n-1).$$
In \cite{Dimitrov2016} we obtain the second-order approximation the Caputo derivative
\begin{equation}\label{1a2}
y^{(\aaaa)}_n= \dddd{1}{\GGGG(2-\aaaa)h^\aaaa}\sum_{k=0}^n \delta_k^{(\aaaa)} y_{n-k}+O\llll(h^{2}\rrrr),
\end{equation}
where $\dddddd_k^{(\aaaa)}=\ssss_k^{(\aaaa)}$ for $2\leq k\leq n$ and
$$\dddddd_0^{(\aaaa)}=\ssss_0^{(\aaaa)}-\zzzz(\aaaa-1), \dddddd_1^{(\aaaa)}=\ssss_1^{(\aaaa)}+2\zzzz(\aaaa-1),\dddddd_2^{(\aaaa)}=\ssss_2^{(\aaaa)}-\zzzz(\aaaa-1).$$
 When $0<\aaaa<1$ and the function $y\in C^2[0,t_n]$, the L1 approximation has an accuracy $O\llll(h^{2-\aaaa} \rrrr)$ and approximation \eqref{1a2} has an accuracy  $O\llll(h^{2} \rrrr)$. The zeta function satisfies $\zzzz(0)=-1/2$. From \eqref{1a2} with $\aaaa=1$ we obtain the second-order approximation for the first derivative
\begin{align}\label{1der}
y'_n=\dddd{1}{h}\llll(\dddd{3}{2}y_n-2y_{n-1}+\dddd{1}{2}y_{n-2}  \rrrr)+O\llll(h^2\rrrr).
\end{align}
The   Caputo  derivative   of order $\aaaa$, where $1<\alpha<2$ is defined as
\begin{equation*}\label{Cder2}
D_C^\alpha y(t)=y^{(\alpha)}(t)=\dddd{d^\aaaa}{d t^\aaaa}y(t)=\dfrac{1}{\Gamma (2-\alpha)}\int_0^t \dfrac{y''(\xi)}{(t-\xi)^{\alpha-1}}d\xi.
\end{equation*}
 In \cite{Dimitrov2017} we obtain the expansion formula of order $4-\aaaa$ of the L1 approximation of the Caputo derivative. When $0<\aaaa<2$, approximation \eqref{1a2} has an asymptotic expansion
\begin{align*}
\dddd{1}{\GGGG(2-\aaaa)h^\aaaa}\sum_{k=0}^n \delta_k^{(\aaaa)} y_{n-k}=y^{(\aaaa)}_n&+\llll(D^2 y_n^{(\aaaa)}-\dddd{y'_0}{\GGGG(-\aaaa)t^{1+\aaaa}} \rrrr)\dddd{h^2}{12}-\\
&\dddd{\zzzz (\aaaa-1)+\zzzz (\aaaa-2)}{\GGGG(2-\aaaa)}y'''_n h^{3-\aaaa}+O\llll(h^{\min \{3,4-\aaaa\}}  \rrrr).
\end{align*}
When $1<\aaaa<2$ approximation \eqref{1a2} has an accuracy $O\llll(h^{3-\aaaa}  \rrrr)$. 
In \cite{Dimitrov2018} we obtain the asymptotic expansion formula
\begin{align*}
\dddd{1}{\GGGG(-\aaaa)h^\aaaa}\llll(\sum_{k=1}^{n-1}
\dddd{y_{n-k}}{k^{1+\aaaa}}-\zzzz(1+\aaaa)y_n\rrrr)=y_n^{(\aaaa)}-&\dddd{\zzzz(\aaaa)}{\GGGG(-\aaaa)}y'_n h^{1-\aaaa}\\
&+\dddd{\zzzz(\aaaa-1)}{2\GGGG(-\aaaa)}y''_n h^{2-\aaaa}+O\llll(h^{3-\aaaa}\rrrr),
\end{align*}
and an approximation of the Caputo derivative
\begin{equation}\label{4a1}
\dddd{1}{\GGGG(-\aaaa)h^\aaaa}\sum_{k=0}^n \gggg_k^{(\aaaa)}z_{n-k}=z_n^{(\aaaa)}+O\llll( h^{3-\aaaa}\rrrr),
\end{equation}
where $\gggg_k^{(\aaaa)}=1/k^{1+\aaaa},(k\geq 3)$ and $\gggg_0^{(\aaaa)}=-\zzzz(1+\aaaa)+\dddd{3}{2} \zzzz(\aaaa)-\dddd{1}{2}\zzzz(\aaaa-1),$
$$\gggg_1^{(\aaaa)}=1-2\zzzz(\aaaa)+ \zzzz(\aaaa-1),\; \gggg_2^{(\aaaa)}=\dddd{1}{2^{1+\aaaa}}+\dddd{1}{2} \zzzz(\aaaa)-\dddd{1}{2}\zzzz(\aaaa-1).$$
 Approximation \eqref{4a1} has an accuracy $O\llll( h^{3-\aaaa}\rrrr)$ when the function $z\in C^3[0,t_n]$ and satisfies $z(0)=z'(0)=z''(0)=0$.

The Miller-Ross sequential  derivative for the  Caputo derivative of order $n\aaaa$ is defined as
 $$y^{[n\aaaa]}(t)=D^{n\aaaa}y(t)=D_C^{\aaaa}D_C^{\aaaa}\cdots D_C^{\aaaa}y(t).$$
The Caputo and Miller-Ross derivatives of order $2\aaaa$ of the function $y(t)=1+t^\aaaa$ satisfy $y^{[2\alpha]}(t)=0$ and $y^{(2\alpha)}(t)=\GGGG(1+\aaaa)t^{-\aaaa}/\GGGG(1-\aaaa)$. The fractional Taylor polynomials of the function $y(t)$ are defined as
$$T_m^{(\aaaa)}(t)= \sum_{k=0}^m\dddd{y^{[k\aaaa]}(0) t^{\aaaa k}}{\GGGG(\aaaa k+1)}.$$
The fractional Taylor polynomials  $T_m^{(\aaaa)}(t)$ (polyfractonomials) are defined at the initial point of fractional differentiation $t=0$ and are polynomials with respect to $t^\aaaa$.
An important class of special functions in fractional calculus are the  one-parameter and two-parameter Mittag-Leffler functions
$$E_\alpha (t)=\sum_{n=0}^\infty \dfrac{t^n}{\Gamma(\alpha n+1)}, \quad
E_{\alpha,\beta} (t)=\sum_{n=0}^\infty \dfrac{t^n}{\Gamma(\alpha n+\beta)}.$$
The Mittag-Leffler functions generalize the exponential function
and appear in the analytical solutions of fractional and integer-order differential equations. The Miller-Ross derivatives of the function $E_\aaaa\llll(t^\aaaa\rrrr)$ satisfy
$$D^{n\aaaa}E_\aaaa\llll(t^\aaaa\rrrr)=E_\aaaa\llll(t^\aaaa\rrrr), \quad \llll. D^{n\aaaa}E_\aaaa\llll(t^\aaaa\rrrr)\rrrr|_{t=0}
=E_\aaaa\llll(0\rrrr)=1.$$
The two-term equation is called fractional relaxation equation, when $0<\aaaa<1$.
\begin{equation} \label{2e1}
y^{(\aaaa)}(t)+By(t)=F(t),\; y(0)=y_0.
\end{equation}
 The exact solution of equation \eqref{2e1} is expressed with the Mittag-Leffler functions as
\begin{equation} \label{21}
y(t)= y_0 \textbf{}E_{\aaaa}(-B t^\aaaa)+\int_{0}^{t}\xi^{\aaaa-1} E_{\aaaa,\aaaa}\llll(-B \xi^\aaaa\rrrr)F(t-\xi)d\xi.
\end{equation}
When the solution of the two-term equation $y\in C^3[0,T]$, the numerical solutions which use approximations \eqref{1a1}, \eqref{1a2} and (4) of the Caputo derivative have accuracy $O\llll(h^{2-\aaaa}\rrrr),\;O\llll(h^{2}\rrrr)$ and $O\llll(h^{3-\aaaa}\rrrr)$. The numerical solutions of FDEs which have smooth solutions have been studied extensively in the last three decades. The finite difference schemes involve  approximations of the Caputo derivative related to the constructions of  the L1 approximation and the Gr\"{u}nwald-Letnikov difference approximation \cite{Alikhanov2015,ChenDeng2014,DingLi2017,GaoSunZhang2014,RenWang2017}.   When $F(t)=0$  the two-term equation \eqref{2e1} has the solution $y(t)=E_{\aaaa}(-B t^\aaaa)$. When $0<\aaaa<1$, the function $E_{\aaaa}(-B t^\aaaa)$ has a singularity at the initial point, because its derivatives are unbounded at $t=0$. This property holds for most ordinary and partial FDEs. The singularity of the solutions of FDEs adds a significant difficulty to the construction of high-order numerical solutions  \cite{JinZhou2015,MurilloYuste2013, ZengZhangKarniadakis2017,ZhangSunLiao2014}. In the present paper we propose a method for transforming linear FDEs  with constant coeeficients into FDEs whose solutions are smooth functions.
 In section 2 and section 3 the method is applied for computing the numerical solutions of the two-term and the three-term FDEs.
\section{Two-term FDE}
 The numerical and analytical solutions of the two-term ordinary FDE  are studied in \cite{Diethelm2010,DiethelmSiegmundTuan2017,Dimitrov2014,Dimitrov2017,Lazarov20151,LiChenYe2011,Podlubny1999}. Let $N$ be a positive integer and $h=T/N$. Suppose that
 \begin{equation} \tag{*}
 \dfrac{1}{h^\alpha }\sum_{k=0}^{n} \LLLL_k^{(\alpha)} y_{n-k}=y^{(\alpha)}_n +O\llll(h^{\bbbb(\aaaa)} \rrrr)
\end{equation}
is an approximatiom of the Caputo derivative of  order $\bbbb (\aaaa)$. Now we drive the numerical solution of two-term equation \eqref{2e1}, which uses approximation (*) of the Caputo derivative. By approximating the Caputo derivative of equation \eqref{2e1} at the point $t_n=nh$ we obtain
 \begin{equation*} 
 \dfrac{1}{h^\alpha }\sum_{k=0}^{n-1} \LLLL_k^{(\alpha)} y_{n-k}+By_n=F_n +O\llll(h^{\bbbb(\aaaa)} \rrrr).
\end{equation*}
 The numerical solution  $\{u_n\}_{n=0}^{N}$ of equation
\eqref{2e1} is computed with $u_0=1$ and
 \begin{equation*} 
 \dfrac{1}{h^\alpha }\sum_{k=0}^{n-1} \LLLL_k^{(\alpha)} u_{n-k}+Bu_n=F_n,
\end{equation*}
\begin{equation}\tag{NS1(*)}
u_n=\dddd{1}{\LLLL_0^{(\alpha)} +B h^\aaaa}\llll(h^{\aaaa}F_n-\sum_{k=1}^{n} \LLLL_k^{(\alpha)} u_{n-k}\rrrr).
\end{equation}
When the solution of the two-term equation $y\in C^3[0,T]$, numerical solutions NS1(1), NS1(2), NS1(4)  have accuracy $O\llll( h^{2-\aaaa}\rrrr),O\llll( h^{2}\rrrr),O\llll( h^{3-\aaaa}\rrrr)$. When the solution of equation \eqref{2e1} satisfies $y(0)=y'(0)=y''(0)$ we can choose the initial values of the numerical solution $u_0=u_1=u_2=0$. The two-term equation
\begin{equation}\label{2e2}
y^{(\aaaa)}(t)+B y(t)=0,\quad y(0)=1
\end{equation}
has the solution $y(t)= E_{\aaaa}(-B t^\aaaa)$. When $0<\aaaa<1$ the first derivative of the solution  is undefined at $t=0$.  Numerical solutions NS1(1) and NS1(2) of equation \eqref{2e2} have accuracy $O\llll( h^\aaaa\rrrr)$ \cite{Lazarov20151}. The numerical results for the error and order of numerical solution NS1(1)  with $\aaaa=0.3,B=1$ and $\aaaa=0.5,B=2$ and NS1(2) with $\aaaa=0.7,B=3$ on the interval $[0,1]$ are presented in Table 1.  The  errors  of the numerical methods in Table 1 and the rest of the tables in the paper are computed with respect to the natural (maximum) $l_\infty$ norm. Now we transform equation \eqref{2e2} into a two-term equation whose solution has a continuous second derivative.
 From  \eqref{2e2} with $t=0$ we obtain $ y^{(\aaaa)}(0)=-B y(0)=-B$. By applying  fractional differentiation of order $\aaaa$ we obtain
\begin{align*}
&y^{[2\aaaa]}(t)+By^{(\aaaa)}(t)=0,\qquad y^{[2\aaaa]}(0)=-By^{(\aaaa)}(0)=B^2,\\
&y^{[3\aaaa]}(t)+By^{[2\aaaa]}(t)=0, \qquad y^{[3\aaaa]}(0)=-By^{[2\aaaa]}(0)=-B^3.
\end{align*}
By induction we obtain the Miller-Ross derivatives of the solution of equation \eqref{2e2} at the initial point $t=0$:
$$y^{[n\aaaa]}(0)=-By^{[(n-1)\aaaa]}(0)=(-B)^n.$$
 Substitute
$$z(t)=y(t)-T_m^{(\aaaa)}(t)=y(t)-\sum_{n=0}^m \dddd{\llll(-B t^{\aaaa}\rrrr)^n}{\GGGG(\aaaa n+1)}.$$
The function $z(t)$ has a Caputo derivative of order $\aaaa$
$$z^{(\aaaa)}(t)=y^{(\aaaa)}(t)+B\sum_{n=0}^{m-1} \dddd{\llll(-B t^{\aaaa}\rrrr)^n}{\GGGG(\aaaa n+1)},$$
 and satisfies the two-term equation
\begin{equation}\label{2e3}
z^{(\aaaa)}(t)+B z(t)=\dddd{(-B)^{m+1} t^{\aaaa m}}{\GGGG(\aaaa m+1)},\quad z(0)=0.
\end{equation}
 Now we use the uniform limit theorem to show that when $m\aaaa>2$  the solution of equation \eqref{2e3} is a twice continuously differentiable function.
\begin{lem} Let $\eeee>0, T>0$ and
$$M>\max\llll\{\dddd{1}{\aaaa}\llll(2B^{1/\aaaa}Te+2-\aaaa \rrrr),\dddd{2}{\aaaa}-1+\log_2\llll(\dddd{2B^{2/\aaaa}}{\eeee}\rrrr)\rrrr\}.$$
Then
$$\llll(\dddd{B^{1/\aaaa}T e}{\aaaa M+\aaaa-2}\rrrr)^\aaaa<\dddd{1}{2}\qquad \text{and}\qquad 2B^{2/\aaaa}\llll(\dddd{B^{1/\aaaa}T e}{\aaaa M+\aaaa-2}\rrrr)^{\aaaa M+\aaaa-2}<\eeee.$$
\end{lem}
\begin{proof}
$$\llll(\dddd{B^{1/\aaaa}T e}{\aaaa M+\aaaa-2}\rrrr)^\aaaa<\dddd{1}{2}\Leftrightarrow \dddd{B^{1/\aaaa}T e}{\aaaa M+\aaaa-2}<\dddd{1}{2^{1/\aaaa}}\Leftrightarrow \aaaa M+\aaaa-2>(2B)^{1/\aaaa}T e.$$
The first inequality is satisfied when
$$M>\dddd{1}{\aaaa}\llll(2B^{1/a}Te+2-\aaaa \rrrr).$$
We have that
$$ \llll(\dddd{B^{1/\aaaa}T e}{\aaaa M+\aaaa-2}\rrrr)^{\aaaa M+\aaaa-2}=\llll(\dddd{B^{1/\aaaa}T e}{\aaaa M+\aaaa-2}\rrrr)^{\aaaa (M+1-2/\aaaa)}<\dddd{1}{2^{M+1-2/\aaaa}}<\dddd{\eeee}{2B^{2/\aaaa}},$$
$$2^{M+1-2/\aaaa}>\dddd{2B^{2/\aaaa}}{\eeee}.$$
The second inequality is satisfied for
$$M>\dddd{2}{\aaaa}-1+\log_2\llll(\dddd{2B^{2/\aaaa}}{\eeee}\rrrr).$$
\end{proof}
\begin{thm} 
Let $m\aaaa>2$ and  $z$ be the solution of equation \eqref{2e3}. Then $z\in C^2[0,T]$.
\end{thm}
\begin{proof} The solution of equation \eqref{2e3} satisfies
$$z(t)=y(t)-T_m^{(\aaaa)}(t)=\sum_{n=m+1}^\infty \dddd{\llll(-B t^{\aaaa}\rrrr)^n}{\GGGG (\aaaa n+1)},$$
$$z''(t)=\sum_{n=m+1}^\infty \dddd{(-B)^n t^{\aaaa n-2}}{\GGGG (\aaaa n-1)}.$$
Denote
$$Z_j(t)=\sum_{n=m+1}^j \dddd{(-B)^n t^{\aaaa n-2}}{\GGGG (\aaaa n-1)},$$
where $j>m$. When $m \aaaa>2$ the functions $Z_j(t)$ are continuous on the interval $[0,T]$. Let $\eeee>0$ and
$$\widetilde{M}>M>\max\llll\{\dddd{1}{\aaaa}\llll(2B^{1/\aaaa}Te+2-\aaaa \rrrr),\dddd{2}{\aaaa}-1+\log_2\llll(\dddd{2B^{2/\aaaa}}{\eeee}\rrrr)\rrrr\}.$$ 
$$Z_{\widetilde{M}}(t)-Z_M(t)=\sum_{n=M+1}^{\widetilde{M}}\dddd{(-B)^n t^{\aaaa n-2}}{\GGGG (\aaaa n-1)}.$$
From the triangle inequality
$$|Z_{\widetilde{M}}(t)-Z_M(t)|<\sum_{n=M+1}^{\widetilde{M}}\dddd{B^n t^{\aaaa n-2}}{\GGGG (\aaaa n-1)}=B^{2/\aaaa}\sum_{n=M+1}^{\widetilde{M}}\dddd{\llll(B^{1/\aaaa}T\rrrr)^{\aaaa n-2}}{\GGGG (\aaaa n-1)}.$$
The gamma function satisfies \cite{Bastero2001}
$$\GGGG (1+x)^{2/x}\geq \dddd{1}{e^2}(x+1)(x+2)>\llll(\dddd{x}{e}\rrrr)^2,$$
$$\GGGG (1+x)>\llll(\dddd{x}{e}\rrrr)^x.$$
Then
$$|Z_{\widetilde{M}}(t)-Z_M(t)|<B^{2/\aaaa}\sum_{n=M+1}^{\widetilde{M}}\dddd{\llll(B^{1/\aaaa}Te \rrrr)^{\aaaa n-2}}{ (\aaaa n-2)^{\aaaa n-2}}<B^{2/\aaaa}\sum_{n=M+1}^{\widetilde{M}}\dddd{\llll(B^{1/\aaaa}Te \rrrr)^{\aaaa n-2}}{ (\aaaa M+\aaaa-2)^{\aaaa n-2}},$$
$$|Z_{\widetilde{M}}(t)-Z_M(t)|<B^{2/\aaaa}\llll(\dddd{B^{1/\aaaa}Te }{\aaaa M+\aaaa-2}\rrrr)^{\aaaa M+\aaaa-2}\sum_{n=0}^{\widetilde{M}}\llll(\dddd{B^{1/\aaaa}Te }{\aaaa M+\aaaa-2}\rrrr)^{\aaaa n}.$$
From Lemma 1:
$$|Z_{\widetilde{M}}(t)-Z_M(t)|<B^{2/\aaaa}\llll(\dddd{B^{1/\aaaa}Te }{\aaaa M+\aaaa-2}\rrrr)^{\aaaa M+\aaaa-2}\sum_{n=0}^{\infty}\dddd{1}{2^n},$$
$$|Z_{\widetilde{M}}(t)-Z_M(t)|<2 B^{2/\aaaa}\llll(\dddd{B^{1/\aaaa}Te }{\aaaa M+\aaaa-2}\rrrr)^{\aaaa M+\aaaa-2}<\eeee.$$
From the uniform limit theorem, the sequence of functions $Z_j(t)$ converges uniformly to the second derivative of the solution of equation \eqref{2e3}, which is a continuous function on the interval $[0,T]$.
\end{proof}
We can show that when $m\aaaa>3$ then $z\in C^3[0,T]$. The proof is similar to the proof of Theorem 2.
In this case numerical solutions NS1(1),NS1(2) and NS1(4) of equation \eqref{2e3} have accuracy $O\llll( h^{2-\aaaa}\rrrr),\;O\llll( h^{2}\rrrr)$ and $O\llll( h^{3-\aaaa}\rrrr)$. 
The numerical results for the   maximum error and order of  numerical solutions NS1(1),NS1(2) and NS1(4) of equation \eqref{2e3} on the inteval $[0,1]$  are presented in Table 2, Table 3 and Table 4.
\section{Three-term  FDE}
In this section we study the numerical solutions of the three-term equation
\begin{equation}\label{4e1}
y^{[2\aaaa]}(t)+3y^{(\aaaa)}(t)+2y(t)=0,\quad y(0)=3,y^{(\aaaa)}(0)=-4,
\end{equation}
where $0<\aaaa<1$. Substitute $w(t)=y^{(\aaaa)}(t)+y(t)$. The function $w(t)$ satisfies the two-term equation
$$w^{(\aaaa)}(t)+2w(t)=0,\quad w(0)=-1.$$
Then $w(t)=-E_\aaaa\llll(-2t^\aaaa  \rrrr)$. The function $y(t)$ satisfies the two-term equation
$$y^{(\aaaa)}(t)+y(t)=-E_\aaaa\llll(-2t^\aaaa  \rrrr),\quad y(0)=3.$$
From \eqref{21}
$$y(t)=3E_\aaaa\llll(-t^\aaaa\rrrr)-\int_0^t \xi^{\aaaa-1}E_{\aaaa,\aaaa}\llll(-\xi^\aaaa\rrrr)E_{\aaaa}\llll(-2(t-\xi)^\aaaa\rrrr)d\xi.$$
From (1.107) in \cite{Podlubny1999} with $\gggg=\aaaa, \bbbb=1, y=-1,z=-2$ we obtain
\begin{align*}
\int_0^t \xi^{\aaaa-1}&E_{\aaaa,\aaaa}\llll(-\xi^\aaaa\rrrr)E_{\aaaa}\llll(-2(t-\xi)^\aaaa\rrrr)d\xi=\llll(2E_{\aaaa,\aaaa+1}\llll(-2t^\aaaa\rrrr)-E_{\aaaa,\aaaa+1}\llll(-t^\aaaa\rrrr)\rrrr)t^\aaaa\\
&=t^\aaaa\llll(2\sum_{n=0}^\infty \dddd{\llll(-2t^{\aaaa }\rrrr)^n}{\GGGG(\aaaa n+\aaaa+1)} -\sum_{n=0}^\infty \dddd{(-t^{\aaaa })^n}{\GGGG(\aaaa n+\aaaa+1)}\rrrr)\qquad [m=n+1]\\
& =-\sum_{m=1}^\infty \dddd{\llll(-2t^{\aaaa }\rrrr)^m}{\GGGG(\aaaa m+1)}+\sum_{m=1}^\infty \dddd{(-t^\aaaa)^m}{\GGGG(\aaaa m+1)}=E_{\aaaa}\llll(-t^\aaaa\rrrr)-E_{\aaaa}\llll(-2t^\aaaa\rrrr).
\end{align*}
The three-term equation \eqref{4e1} has the solution $ y(t)=2E_{\aaaa}\llll(-t^\aaaa\rrrr)+E_{\aaaa}\llll(-2t^\aaaa\rrrr)$. When $0<\aaaa<1$ the first derivative of the solution $y(t)$ of equation \eqref{4e1} is undefined at the initial point $t=0$. The Riemann-Liouville derivative of order $\aaaa$, where $n-1\leq \aaaa<n$ and $n$ is a nonnegative integer is defined as
$$D_{RL}^\aaaa y(t)=\dddd{1}{\GGGG(n-\aaaa)}\dddd{d^n}{dt^n}\int_0^t\dddd{y(\xi)}{(t-\xi)^{1+\aaaa-n}}d\xi.
$$ The Caputo, Miller-Ross and Rieman-Lioville derivatives satisfy \cite{Diethelm2010,Podlubny1999}:
$$D^{\aaaa} y(t)=D_{C}^{\aaaa}y(t)=D_{RL}^{\aaaa} y(t)-\dddd{y(0)}{\GGGG(1-\aaaa)t^{\aaaa}},$$
$$D_{RL}^{2\aaaa} y(t)=D_{C}^{2\aaaa} y(t)+\dddd{y(0)}{\GGGG(1-2\aaaa)t^{2\aaaa}}=D^{2\aaaa} y(t)+\dddd{y(0)}{\GGGG(1-2\aaaa)t^{2\aaaa}}+\dddd{y^{(\aaaa)}(0)}{\GGGG(1-\aaaa)t^{\aaaa}},
$$
where the above identity for the Caputo derivative requires that $0<\aaaa<0.5$.
 Three-term equation \eqref{4e1} is reformulated with  the Riemann-Liouville and Caputo fractional derivatives as
\begin{equation}\label{ERLDF}
D_{RL}^{2\aaaa} y(t)+3D_{RL}^{\aaaa} y(t)+2y(t)=\dddd{3}{\GGGG(1-2\aaaa)t^{2\aaaa}}+\dddd{5}{\GGGG(1-\aaaa)t^{\aaaa}},
\end{equation}
\begin{equation}\label{ECDF}
D_{C}^{2\aaaa} y(t)+3D_{C}^{\aaaa} y(t)+2y(t)=-\dddd{4}{\GGGG(1-\aaaa)t^{\aaaa}}, y(0)=3,\;  0<\aaaa\leq 0.5.
\end{equation}
  A similar three-term equation  which involves the Caputo derivative is studied in Example 5.1 in \cite{ZengZhangKarniadakis2017}.
 The formulation \eqref{ECDF} of the three-term equation which uses the Caputo derivative has only one initial condition specified and the values of $\aaaa$, where $0.5<\aaaa<1$ are excluded. The initial conditions of equations \eqref{ERLDF} and \eqref{ECDF} are inferred from the equations. Formulation \eqref{4e1} of the three-term equation studied in this section has the advantages, to the formulations  \eqref{ERLDF} and \eqref{ECDF} which use the Caputo and Riemann-Lioville derivatives, that the initial conditions of equation \eqref{4e1} are specified and the analytical solution is obtained with the method described in this section. 
 Now we determine the Miller-Ross derivatives of the solution of equation \eqref{4e1}. By applying  fractional differentiation of order $\aaaa$ we obtain
$$y^{[(n+1)\aaaa]}(t)+3y^{[n\aaaa]}(t)+2y^{[(n-1)\aaaa]}(t)=0.$$
Denote $a_n=y^{[n\aaaa]}(0)$. The numbers $a_n$ are computed recursively with
$$a_{n+1}+3a_n+2a_{n-1}=0,\quad a_0=3,a_1=-4.$$
Substitute
$$z(t)=y(t)-T_m^{(\aaaa)}(t)=y(t)-\sum_{n=0}^m \dddd{a_n t^{n\aaaa}}{\GGGG(\aaaa n+1)}.$$
The function $z(t)$ has fractional derivatives
$$z^{(\aaaa)}(t)=y^{(\aaaa)}(t)-\sum_{n=1}^m \dddd{a_n t^{(n-1)\aaaa}}{\GGGG(\aaaa (n-1)+1)}=y^{(\aaaa)}(t)-\sum_{n=0}^{m-1} \dddd{a_{n+1}t^{n\aaaa}}{\GGGG(\aaaa n+1)},$$
$$z^{[2\aaaa]}(t)=y^{[2\aaaa]}(t)-\sum_{n=1}^{m-1} \dddd{a_{n+1}t^{(n-1)\aaaa}}{\GGGG(\aaaa (n-1)+1)}=y^{[2\aaaa]}(t)-\sum_{n=0}^{m-2} \dddd{a_{n+2}t^{n\aaaa}}{\GGGG(\aaaa n+1)}.$$
 The function $z(t)$ satisfies the condition $z(0)=z^{(\aaaa)}(0)=0$ and its  Miller-Ross and Caputo derivatives of order $2\aaaa$  are equal $z^{[2\aaaa]}(t)=z^{(2\aaaa)}(t)$.  The function $z(t)$ satisfies the three-term equation
\begin{equation}\label{4e2}
z^{(2\aaaa)}(t)+3z^{(\aaaa)}(t)+2z(t)=F(t),\quad z(0)=0,z^{(\aaaa)}(0)=0,
\end{equation}
where
$$F(t)=-\llll(2a_{m-1}+3a_{m} \rrrr)\dddd{t^{(m-1)\aaaa}}{\GGGG(\aaaa (m-1)+1)}-\dddd{2a_{m}t^{m\aaaa}}{\GGGG(\aaaa m+1)}.$$
Now we obtain the numerical solution of three-term equation \eqref{4e2} which uses approximation (*) of the Caputo derivative. By approximating the Caputo derivatives of equation \eqref{4e2} at the point $t_n=n h$ we obtain
\begin{equation}\label{42}
\dddd{1}{h^{2\aaaa}}\sum_{k=0}^{n-1} \LLLL_k^{(2\alpha)} z_{n-k}+\dddd{3}{h^{\aaaa}} \sum_{k=0}^{n-1} \LLLL_k^{(\alpha)} z_{n-k}+2z_n=F_n+O\llll(h^{\min\{\bbbb(\aaaa),\bbbb(2\aaaa)\}} \rrrr).
\end{equation}
Denote by NS2(*) the numerical solution $\{u_n\}_{n=0}^N$ of equation \eqref{4e2} which uses approximation (*) of the Caputo derivative. From \eqref{42}
$$\LLLL_0^{(2\alpha)} u_{n}+3h^\aaaa\LLLL_0^{(\alpha)} u_{n}+2h^{2\aaaa}u_n+\sum_{k=1}^{n-1} \LLLL_k^{(2\alpha)} u_{n-k}+3h^\aaaa \sum_{k=1}^{n-1} \LLLL_k^{(\alpha)} u_{n-k}=h^{2\aaaa}F_n.$$
The numbers $u_n$ are computed with $u_0=u_1=0$ and
\begin{equation}\tag{NS2(*)}
u_n=\dddd{1}{\LLLL_0^{(2\alpha)} +3h^\aaaa\LLLL_0^{(\alpha)} +2h^{2\aaaa}}\llll(h^{2\aaaa}F_n-\sum_{k=1}^{n-1} \llll(\LLLL_k^{(2\alpha)} - 3h^\aaaa \LLLL_k^{(\alpha)}\rrrr) u_{n-k}\rrrr).
\end{equation}

 When $y\in C^3[0,T]$ numerical solution NS2(2), which uses approximation \eqref{1a2} of the Caputo derivative has an order $\min\{2,3-2\aaaa\}$. The accuracy of numerical solution NS2(2) is $O\llll(h^2\rrrr)$ when $0<\aaaa\leq 0.5$ and  $O\llll(h^{3-2\aaaa}\rrrr)$ for $0.5<\aaaa<1$ .
The numerical results for the error and the order of numerical solution NS2(2) of three-term equation \eqref{4e2} on the interval $[0,1]$ with are presented in Table 5.
 Numerical solution NS2(4) which uses approximation \eqref{4a1} of the Caputo derivative has an accuracy $O\llll( h^{3-2\aaaa}\rrrr)$ for all values of $\aaaa \in (0,0.5)\cup (0.5,1)$.
Numerical solution  NS2(4) is undefined for $\aaaa=0.5$, because the value of $\GGGG(-2\aaaa)=\GGGG(-1)$ is undefined.
The numerical results for the error and the order of numerical solution NS2(4) of equation \eqref{4e2}  on the interval $[0,1]$   are presented in Table 6. Now we obtain the numerical solution NS3(*) of  equation \eqref{4e2} with $\aaaa=0.5$ which uses approximation (*) for the Caputo derivative.
\begin{equation}\label{4e3}
z'(t)+3z^{(0.5)}(t)+2z(t)=F(t).
\end{equation}
By approximating the first derivative with \eqref{1der}  we obtain
\begin{align*}
\dddd{1}{h}\llll(\dddd{3}{2}z_n-2z_{n-1}+\dddd{1}{2}z_{n-2}  \rrrr)+\dfrac{3}{h^{0.5} }\sum_{k=0}^{n} \LLLL_k^{(0.5)} z_{n-k}+2z_n=F_n+O\llll(h^{\min\{2,\bbbb (0.5)\}}\rrrr).
\end{align*}
The numerical solution $\{u_n\}_{n=0}^N$ of equation \eqref{4e3} is computed with $u_0=u_1=0$ and
\begin{align*}
\dddd{1}{h}\llll(\dddd{3}{2}u_n-2u_{n-1}+\dddd{1}{2}u_{n-2}  \rrrr)+\dfrac{3}{h^{0.5} }\sum_{k=0}^{n} \LLLL_k^{(0.5)} u_{n-k}+2u_n=F_n,
\end{align*}
\begin{equation}\tag{NS3(*)}
u_n=\dddd{1}{3+6\LLLL_0^{(0.5)}h^{0.5}+4h}\llll(2hF_n+4y_{n-1}-y_{n-2}-6h^{0.5} \sum_{k=1}^{n} \LLLL_k^{(0.5)} u_{n-k} \rrrr)
\end{equation}
The numerical results for the error and the order of numerical solutions NS3(1), NS3(2) and NS3(4) of equation \eqref{4e3} are presented in Table 7.
\section{Conclusion} In the present paper we propose a method for improving the numerical solutions of ordinary fractional differential equations.  The method is based on the computation of the fractional Taylor polynomials of the solution at the initial point and transforming the equations into FDEs which have smooth solutions. The method is used for computing the numerical solutions of equations \eqref{2e2}, 
 and \eqref{4e1} and it can be applied to other linear fractional differential equations with constant coefficients.  The fractional Taylor polynomials have a potential for construction of numerical solutions of linear and nonlinear  FDEs which have singular solutions.   In future work we are going to generalize the method discussed in the paper to other linear FDEs with constant coefficients and analyze the convergence and the stability of the numerical methods.
\section{Acknowledgements}
This work was supported by the Bulgarian Academy of Sciences through the Program for Career Development of Young Scientists, Grant DFNP-17-88/2017, Project "Efficient Numerical Methods with an Improved Rate of Convergence for Applied Computational Problems'', by the Bulgarian National Fund of Science under Project DN 12/4-2017 "Advanced Analytical and Numerical Methods for Nonlinear Differential Equations with Applications in Finance and Environmental Pollution" and by the Bulgarian National Fund of Science under Project DN 12/5-2017 "Efficient Stochastic Methods and Algorithms for Large-Scale Computational Problems''.

\newpage
	\begin{table}[ht]
	\centering
    \caption{Maximum error and order of numerical solutions NS1(1)   of  equation \eqref{2e2} with $\aaaa=0.3,\aaaa=0.5$  and NS1(2) with $\aaaa=0.7$ of order $\aaaa$. }
\small
  \begin{tabular}{ |l | c  c | c  c | c  c| }
		\hline
		\multirow{2}*{ $\quad \boldsymbol{h}$}  & \multicolumn{2}{c|}{$\boldsymbol{\aaaa=0.3,B=1}$} & \multicolumn{2}{c|}{$\boldsymbol{\aaaa=0.5,B=2}$}  & \multicolumn{2}{c|}{$\boldsymbol{\aaaa=0.7,B=3}$} \\
		\cline{2-7}
   & $Error$ & $Order$  & $Error$ & $Order$  & $Error$ & $Order$ \\
		\hline
$0.00625$    & $0.3344\times 10^{-1}$  & $0.2089$      & $0.9063\times 10^{-1}$  & $0.3987$       & $0.6385\times 10^{-1}$   & $0.6332$        \\
$0.003125$   & $0.2863\times 10^{-1}$  & $0.2242$      & $0.6743\times 10^{-1}$  & $0.4265$       & $0.4046\times 10^{-1}$   & $0.6583$        \\
$0.0015625$  & $0.2429\times 10^{-1}$  & $0.2372$      & $0.4946\times 10^{-1}$  & $0.4471$       & $0.2536\times 10^{-1}$   & $0.6741$        \\
$0.00078125$ & $0.2045\times 10^{-1}$  & $0.2482$      & $0.3591\times 10^{-1}$  & $0.4621$       & $0.1578\times 10^{-1}$   & $0.6840$        \\
\hline
  \end{tabular}
\end{table}
	\begin{table}[ht]
	\centering
    \caption{Maximum error and order of numerical solution NS1(1)   of  equation \eqref{2e3} of order $2-\aaaa$. }
\small
  \begin{tabular}{ |l | c  c | c  c | c  c| }
		\hline
		\multirow{2}*{ $\quad \boldsymbol{h}$}  & \multicolumn{2}{c|}{$\boldsymbol{\aaaa=0.3,B=1,m=7}$} & \multicolumn{2}{c|}{$\boldsymbol{\aaaa=0.5,B=2,m=4}$}  & \multicolumn{2}{c|}{$\boldsymbol{\aaaa=0.7,B=3,m=3}$} \\
		\cline{2-7}
   & $Error$ & $Order$  & $Error$ & $Order$  & $Error$ & $Order$ \\
		\hline
$0.00625$    & $0.7428\times 10^{-7}$  & $1.6596$  & $0.4811\times 10^{-3}$  & $1.4858$    & $0.2821\times 10^{-2}$   & $1.2943$  \\
$0.003125$   & $0.2338\times 10^{-7}$  & $1.6681$  & $0.1713\times 10^{-3}$  & $1.4901$    & $0.1148\times 10^{-2}$   & $1.2966$  \\
$0.0015625$  & $0.7322\times 10^{-8}$  & $1.6746$  & $0.6085\times 10^{-4}$  & $1.4931$    & $0.4671\times 10^{-3}$   & $1.2980$  \\
$0.00078125$ & $0.2286\times 10^{-8}$  & $1.6797$  & $0.2159\times 10^{-4}$  & $1.4951$    & $0.1899\times 10^{-3}$   & $1.2988$  \\
\hline
  \end{tabular}
\end{table}
	\begin{table}[ht]
	\centering
    \caption{Maximum error and order of second-order numerical solution NS1(2)   of  equation \eqref{2e3}. }
\small
  \begin{tabular}{ |l | c  c | c  c | c  c| }
		\hline
		\multirow{2}*{ $\quad \boldsymbol{h}$}  & \multicolumn{2}{c|}{$\boldsymbol{\aaaa=0.3,B=1,m=8}$} & \multicolumn{2}{c|}{$\boldsymbol{\aaaa=0.5,B=2,m=5}$}  & \multicolumn{2}{c|}{$\boldsymbol{\aaaa=0.7,B=3,m=2}$} \\
		\cline{2-7}
   & $Error$ & $Order$  & $Error$ & $Order$  & $Error$ & $Order$ \\
		\hline
$0.00625$    & $0.9068\times 10^{-6}$  & $1.9692$   & $0.2333\times 10^{-4}$  & $1.8861$   & $0.2789\times 10^{-3}$ & $2.0586$  \\
$0.003125$   & $0.2297\times 10^{-6}$  & $1.9809$   & $0.6148\times 10^{-5}$  & $1.9240$   & $0.6715\times 10^{-4}$ & $2.0545$  \\
$0.0015625$  & $0.5790\times 10^{-7}$  & $1.9882$   & $0.1593\times 10^{-5}$  & $1.9484$   & $0.1597\times 10^{-4}$ & $2.0718$ \\
$0.00078125$ & $0.1455\times 10^{-7}$  & $1.9927$   & $0.4081\times 10^{-6}$  & $1.9645$   & $0.3771\times 10^{-5}$ & $2.0826$ \\
\hline
  \end{tabular}
\end{table}
	\begin{table}[ht]
	\centering
    \caption{Maximum error and order of numerical solution NS1(4)   of  equation \eqref{2e3} of order $3-\aaaa$. }
\small
\vspace{-0.2cm}
  \begin{tabular}{ |l | c  c | c  c | c  c| }
		\hline
		\multirow{2}*{ $\quad \boldsymbol{h}$}  & \multicolumn{2}{c|}{$\boldsymbol{\aaaa=0.3,B=1,m=8}$} & \multicolumn{2}{c|}{$\boldsymbol{\aaaa=0.5,B=2,m=6}$}  & \multicolumn{2}{c|}{$\boldsymbol{\aaaa=0.7,B=3,m=5}$} \\
		\cline{2-7}
   & $Error$ & $Order$  & $Error$ & $Order$  & $Error$ & $Order$ \\
		\hline
$0.00625$    & $0.1468\times 10^{-5}$  & $2.6281$ & $0.5705\times 10^{-5}$  & $2.4902$ &$0.8051\times 10^{-4}$ & $2.2933$ \\
$0.003125$   & $0.2329\times 10^{-6}$  & $2.6565$ & $0.1011\times 10^{-5}$  & $2.4963$ &$0.1638\times 10^{-4}$ & $2.2968$ \\
$0.0015625$  & $0.3675\times 10^{-7}$  & $2.6637$ & $0.1789\times 10^{-6}$  & $2.4985$ &$0.3331\times 10^{-5}$ & $2.2984$ \\
$0.00078125$ & $0.5776\times 10^{-8}$  & $2.6699$ & $0.3164\times 10^{-7}$  & $2.4995$ &$0.6767\times 10^{-6}$ & $2.2992$ \\
\hline
  \end{tabular}
\end{table}
\begin{table}[ht]
	\caption{Maximum error and order of numerical solution NS2(2) of  equation \eqref{4e2} of order $\min\{2,3-2\aaaa\}$.}
	\small
	\centering
  \begin{tabular}{ |l | c  c | c  c | c  c| }
		\hline
		\multirow{2}*{ $\quad \boldsymbol{h}$}  & \multicolumn{2}{c|}{$\boldsymbol{\aaaa=0.3,m=9}$} & \multicolumn{2}{c|}{$\boldsymbol{\aaaa=0.4,m=6}$}  & \multicolumn{2}{c|}{$\boldsymbol{\aaaa=0.7,m=5}$} \\
		\cline{2-7}
   & $Error$ & $Order$  & $Error$ & $Order$  & $Error$ & $Order$ \\
		\hline
$0.00625$    & $0.7315\times 10^{-3}$ & $1.9149$ & $0.7261\times 10^{-4}$  & $1.8455$ & $0.5906\times 10^{-3}$ & $1.5957$ \\
$0.003125$   & $0.1902\times 10^{-3}$ & $1.9429$ & $0.1968\times 10^{-4}$  & $1.8838$ & $0.1956\times 10^{-3}$ & $1.5945$ \\
$0.0015625$  & $0.4886\times 10^{-4}$ & $1.9610$ & $0.5231\times 10^{-5}$  & $1.9110$ & $0.6478\times 10^{-4}$ & $1.5941$ \\
$0.00078125$ & $0.1245\times 10^{-4}$ & $1.9728$ & $0.1372\times 10^{-5}$  & $1.9306$ & $0.2145\times 10^{-4}$ & $1.5943$ \\
\hline
  \end{tabular}
	\end{table}	
\begin{table}[ht]
	\caption{Maximum error and order of numerical solution NS2(4) of  equation \eqref{4e2} of order $3-2\aaaa$.}
	\small
	\centering
	\vspace{-0.2cm}
  \begin{tabular}{ |l | c  c | c  c | c  c| }
		\hline
		\multirow{2}*{ $\quad \boldsymbol{h}$}  & \multicolumn{2}{c|}{$\boldsymbol{\aaaa=0.3,m=14}$} & \multicolumn{2}{c|}{$\boldsymbol{\aaaa=0.4,m=10}$}  & \multicolumn{2}{c|}{$\boldsymbol{\aaaa=0.7,m=6}$} \\
		\cline{2-7}
   & $Error$ & $Order$  & $Error$ & $Order$  & $Error$ & $Order$ \\
		\hline
$0.00625$    & $0.1202\times 10^{-2}$ & $2.4521$ & $0.3205\times 10^{-3}$ & $2.2431$ & $0.5709\times 10^{-3}$ & $1.6080$\\
$0.003125$   & $0.2202\times 10^{-3}$ & $2.4481$ & $0.6799\times 10^{-4}$ & $2.2369$ & $0.1874\times 10^{-3}$ & $1.6069$\\
$0.0015625$  & $0.4052\times 10^{-4}$ & $2.4423$ & $0.1449\times 10^{-4}$ & $2.2303$ & $0.6161\times 10^{-4}$ & $1.6052$\\
$0.00078125$ & $0.7486\times 10^{-5}$ & $2.4363$ & $0.3101\times 10^{-5}$ & $2.2242$ & $0.2027\times 10^{-4}$ & $1.6037$\\
\hline
  \end{tabular}
	\end{table}
		\vspace{-0.5cm}
	\begin{table}[ht]
	\caption{Maximum error and order of numerical solution NS3(1)  of  equation \eqref{4e3} of order $1.5$ and second-order numerical solutions NS3(2) ana NS3(4).}
	\small
	\centering
  \begin{tabular}{ |l | c  c | c  c | c  c| }
		\hline
		\multirow{2}*{ $\quad \boldsymbol{h}$}  & \multicolumn{2}{c|}{$\boldsymbol{NS3(1),m=4}$} & \multicolumn{2}{c|}{$\boldsymbol{NS3(2),m=3}$}  & \multicolumn{2}{c|}{$\boldsymbol{NS3(4),m=5}$} \\
		\cline{2-7}
   & $Error$ & $Order$  & $Error$ & $Order$  & $Error$ & $Order$ \\
		\hline
$0.00625$    & $0.8076\times 10^{-3}$ & $1.4835$ & $0.4197\times 10^{-3}$ & $1.8868$ & $0.6085\times 10^{-4}$ & $2.0257$\\
$0.003125$   & $0.2872\times 10^{-3}$ & $1.4915$ & $0.1109\times 10^{-3}$ & $1.9195$ & $0.1498\times 10^{-4}$ & $2.0224$\\
$0.0015625$  & $0.1019\times 10^{-3}$ & $1.4955$ & $0.2886\times 10^{-4}$ & $1.9429$ & $0.3698\times 10^{-5}$ & $2.0179$\\
$0.00078125$ & $0.3608\times 10^{-4}$ & $1.4975$ & $0.7473\times 10^{-5}$ & $1.9491$ & $0.9157\times 10^{-6}$ & $2.0138$\\
\hline
  \end{tabular}
	\end{table}	
	\clearpage
	\noindent
Y. Dimitrov \\
 Department of Mathematics and Physics, University of Forestry, Sofia  1756, Bulgaria\\
{\it yuri.dimitrov@ltu.bg}\\
I. Dimov \\
  Institute of Information and Communication Technologies, Bulgarian Academy of Sciences,\\
	Department of Parallel Algorithms, Acad. Georgi Bonchev Str., Block 25 A, 1113 Sofia, Bulgaria,\\
{\it ivdimov@bas.bg}\\
 V. Todorov \\
Institute of Mathematics and Informatics, Bulgarian Academy of Sciences, Department of Information Modeling, Acad. Georgi Bonchev Str., Block 8, 1113 Sofia, Bulgaria,\\
{\it vtodorov@math.bas.bg}\\
 Institute of Information and Communication Technologies, Bulgarian Academy of Sciences,\\ 
Department of Parallel Algorithms,
 Acad. Georgi Bonchev Str., Block 25 A, 1113 Sofia, Bulgaria\\
{\it venelin@parallel.bas.bg}\\

\end{document}